\documentclass{amsart}
\usepackage{graphicx}

\newcommand{\esssup}{\mathop\mathrm{ess\,sup}}
\newcommand{\length}{\mathcal{H}^1}
\newcommand{\surf}{\mathcal{H}^{n-1}}
\newcommand{\HD}{\mathop\mathrm{HD}}
\newcommand{\tK}{\widetilde{K}}
\newcommand{\tH}{\widetilde{H}}
\newcommand{\RR}{\mathbb{R}}
\newcommand{\dist}{\mathop\mathrm{dist}}
\newcommand{\Mod}{\mathrm{mod}_n}

\theoremstyle{theorem}
\newtheorem{theorem}{Theorem}[section]
\newtheorem{corollary}[theorem]{Corollary}
\newtheorem{lemma}[theorem]{Lemma}

\newtheoremstyle{defnbreak}
  {}
  {}
  {}
  {}
  {\bfseries}
  {.}
  {\parindent}
  {}

\theoremstyle{defnbreak}
\newtheorem{definition}[theorem]{Definition}

\theoremstyle{remark}
\newtheorem{remark}[theorem]{Remark}

\numberwithin{equation}{section}

\begin{document}
\date{July 31, 2012}
\title{Quasisymmetry and Rectifiability of Quasispheres}
\author{Matthew Badger}
\address{Department of Mathematics\\ Stony Brook University \\ Stony Brook, NY 11794-3651}
\email{badger@math.sunysb.edu}
\author{James T. Gill}
\address{Department of Mathematics and Computer Science\\ Saint Louis University\\ St.~ Louis, MO  63103}
\email{jgill5@slu.edu}
\author{Steffen Rohde}
\address{Department of Mathematics\\ University of Washington\\ Seattle, WA 98195-4350}
\email{rohde@math.washington.edu}
\author{Tatiana Toro}
\address{Department of Mathematics\\ University of Washington\\ Seattle, WA 98195-4350}
\email{toro@math.washington.edu}
\thanks{The authors were partially supported by the following NSF grants: the first author, \#0838212; the second author, \#1004721; the third author, \#0800968 and \#1068105; and the fourth author, \#0856687.}

\begin{abstract} We obtain Dini conditions that guarantee that an asymptotically conformal quasisphere is rectifiable. In particular, we show that for any $\epsilon>0$ integrability of  $( {\rm ess}\sup_{1-t<|x|<1+t} K_f(x)-1 )^{2-\epsilon} dt/t$
implies that the image of the unit sphere under a global quasiconformal homeomorphism $f$ is rectifiable. We also establish estimates for the weak quasisymmetry constant of a global $K$-quasiconformal map in neighborhoods with maximal dilatation close to 1.
\end{abstract}

\maketitle

\section{Introduction}

A \emph{quasisphere} $f(S^{n-1})$ is the image of the unit sphere $S^{n-1}\subset\RR^n$ under a global quasiconformal mapping $f:\RR^n\rightarrow\RR^n$. In the plane, a quasisphere is a \emph{quasicircle}.
(Look below for the definition of a quasiconformal map.) It is well known that the Hausdorff dimension of a quasisphere can exceed $n-1$. When $n=2$, for example, the von Koch snowflake is a quasicircle with dimension $\log 4/\log 3>1$. In fact, every quasicircle is bi-Lipschitz equivalent to a snowflake-like curve (see Rohde \cite{R}). On the other hand, the Hausdorff dimension of a quasicircle cannot be too large: Smirnov \cite{Sm} proved Astala's conjecture that every $k$-quasicircle ($0\leq k<1$) has dimension at most $1+k^2$. This result was further enhanced by Prause, Tolsa and Uriarte-Tuero \cite{PTU}  who showed that $k$-quasi\-circles have finite $(1+k^2)$-dimensional Hausdorff measure. The picture in higher dimensions is not as complete. A few detailed examples of quasispheres with dimension greater than $n-1$ ($n\geq 3$)  have been described by David and Toro \cite{DT1} and Meyer \cite{M1}, \cite{M2}. Mattila and Vuorinen \cite{MV} have also demonstrated how the maximal dilatation (see (\ref{maxdilatation})) of a quasiconformal map $f$ controls the geometry and size of the quasisphere $f(S^{n-1})$.
More specifically, they showed that if $f$ is $K$-quasiconformal with $K$ near 1,
then $f(S^{n-1})$ satisfies the \emph{linear approximation property} (see \cite{MV}) and this property bounds the dimension of $f(S^{n-1})$.
Mattila and Vuorinen's proof that quasispheres are locally uniformly approximable by hyper\-planes
was recently streamlined by Prause \cite{P}, using the quasisymmetry of $f$. This idea from \cite{P} will play an important role in our analysis below.

In the current article, we seek optimal conditions on $f$ that ensure $f(S^{n-1})$ has finite $(n-1)$-dimensional Hausdorff measure $\surf$. We obtain two such conditions, one expressed in terms of the dilatation of $f$ (Theorem \ref{ThmB}) and one expressed in terms of the quasisymmetry of $f$ (Theorem \ref{ThmA}), and both have sharp exponent. This problem was previously studied in the case $n=2$ by Anderson, Becker and Lesley \cite{ABL} and in all dimensions by Mattila and Vuorinen \cite{MV}. To state these results and the main results of this paper, we require some additional notation.

Let $1\leq K<\infty$. A mapping $f:\Omega\rightarrow\RR^n$ from a domain $\Omega\subset\RR^n$ ($n\geq 2$) is said to be \emph{$K$-quasiconformal} (analytic definition) if $f\in W^{1,n}_\mathrm{loc}(\Omega)$, if $f$ is a homeomorphism onto its image, and if the \emph{maximal dilatation} $K_f(\Omega)$ is bounded by $K$: \begin{equation}\label{maxdilatation}K_f(\Omega)=\max\left\{\esssup_{x\in\Omega}\frac{\|f'(x)\|^n}{|Jf(x)|}, \esssup_{x\in\Omega}\frac{|Jf(x)|}{\ell(f'(x))^n}\right\}\leq K.\end{equation} Here we let $f'$ and $Jf$ denote the Jacobian matrix and Jacobian determinant of $f$, respectively. Also $\|\cdot\|$ denotes the operator norm and $\ell(f'(x))=\min_{\|v\|=1}|f'(x)v|$. For background on quasiconformal maps in higher dimensions, we refer the reader to V\"ais\"al\"a \cite{V} and Heinonen \cite{H}. For $t>0$, set $A_t=\{x\in\RR^n:1-t< |x|< 1+t\}$, the annular neighborhood of $S^{n-1}$ of size $t$. We say that a quasisphere $f(S^{n-1})$ is \emph{asymptotically conformal} if $K_f(A_t)\rightarrow 1$ as $t\rightarrow 0$. It will be convenient to also introduce the notation \begin{equation}\tK_f(\Omega)=K_f(\Omega)-1.\end{equation} Notice that $K_f(A_t)\rightarrow 1$ as $t\rightarrow 0$ if and only if $\tK_f(A_t)\rightarrow 0$ as $t\rightarrow 0$.

Every asymptotically conformal quasisphere $f(S^{n-1})$ has Hausdorff dimension $n-1$; see Remark \ref{asympdim}. This is the best that we can do in general, because there are snowflake-like curves $\Gamma=f(S^1)$ such that $\tK_f(A_t)\rightarrow 0$ as $t\rightarrow 0$ but $\length(\Gamma)=\infty$. The main obstruction to finite Hausdorff measure is that $\tK_f(A_t)$ could converge to $0$ very slowly as $t\rightarrow 0$. Conversely, one expects that a good rate of convergence should guarantee that $\surf(f(S^{n-1}))<\infty$. One would like to determine the threshold for a ``good rate". In \cite{ABL} Anderson, Becker and Lesley proved that (in our notation) if $f:\RR^2\rightarrow\RR^2$ is quasiconformal and $f|_{B(0,1)}$ is conformal, then \begin{equation}\label{ABLeq}\int_0^1 \left(\tK_f(A_t)\right)^2 \frac{dt}{t}<\infty\quad\Longrightarrow\quad \length(f(S^1))<\infty.\end{equation} It is also known that the exponent $2$ in the Dini condition (\ref{ABLeq}) cannot be weakened to $2+\epsilon$ for any $\epsilon>0$ (see \cite{ABL}). In higher dimensions, Mattila and Vuorinen \cite{MV} proved that (in our notation) \begin{equation}\label{MVeq} \int_0^1 \tK_f(A_t)\frac{dt}{t}<\infty\quad\Longrightarrow\quad f|_{S^{n-1}}\text{ is Lipschitz.}\end{equation} Hence, by  a standard property of Lipschitz maps,  the Dini condition in (\ref{MVeq}) also implies that the quasisphere $f(S^{n-1})$ is $(n-1)$-rectifiable and  $\surf(f(S^{n-1}))<\infty$. In fact, something more is true: the Dini condition $(\ref{MVeq})$ implies that $f(S^{n-1})$ is a $C^1$ submanifold of $\RR^n$ (see Chapter 7, \S4 in Reshetnyak \cite{Resh}). For conditions weaker than (\ref{MVeq}) (but also with ``exponent 1") that imply $f|_{S^{n-1}}$ is Lipschitz, see Bishop, Gutlyanski\u\i, Martio and Vuorinen \cite{BGMV} and Gutlyanski\u{\i} and Golberg \cite{GG}. Notice that the Dini condition (\ref{MVeq}) is stronger (harder to satisfy) than the Dini condition (\ref{ABLeq}). The main result of this paper is that a Dini condition with exponent 2 ensures that $\surf(f(S^{n-1}))<\infty$ in dimensions $n\geq 3$, and moreover guarantees the existence of local bi-Lipschitz parameterizations.

\begin{theorem}\label{ThmB} If $f:\RR^n\rightarrow\RR^n$ is quasiconformal and \begin{equation}\label{Kfinite}\int_0^1 \left(\tK_f(A_t) \log\frac{1}{\tK_f(A_t)}\right)^2\frac{dt}{t}<\infty,\end{equation} then the quasisphere $f(S^{n-1})$ admits local $(1+\delta)$-bi-Lipschitz parameterizations, for every $\delta>0$. Thus $f(S^{n-1})$ is $(n-1)$-rectifiable and $\surf(f(S^{n-1}))<\infty$.\end{theorem}

The main difference between Mattila and Vuorinen's theorem and Theorem \ref{ThmB} is that the former is a statement about the regularity of $f|_{S^{n-1}}$, while the latter is a statement about the regularity of $f(S^{n-1})$.  The logarithmic term in (\ref{Kfinite}) is an artifact from the proof of Theorem 1.1, which occurs when we use the maximal dilatation of the map $f$ to control the weak quasisymmetry constant of $f$ (see (\ref{Hdefn})). We do not know whether this term can be removed, and leave this open for future investigation. Nevertheless, Theorem \ref{ThmB} has the following immediate consequence. If $f:\RR^n\rightarrow\RR^n$ is quasiconformal and \begin{equation}\label{Kcor}\int_0^1 \left(\tK_f(A_t)\right)^{2-\epsilon}\frac{dt}{t}<\infty\quad\text{for some }\epsilon>0,\end{equation} then $f$ satisfies (\ref{Kfinite}), and in particular, the quasisphere $f(S^{n-1})$ satisfies the same conclusions as in Theorem \ref{ThmB}. The exponent 2 in Theorem \ref{ThmB} is the best possible, i.e.\ 2 cannot be replaced with $2+\epsilon$ for any $\epsilon>0$. For example, the construction in David and Toro \cite{DT1} (with the parameters $Z=\RR^{n-1}$ and $\epsilon_j=1/j$) can be used to produce a quasiconformal map $f:\RR^n\rightarrow\RR^n$ such that \begin{equation} \int_0^1 \left(\tK_f(\{x\in\RR^n:|x_n|<t\})\right)^{2+\epsilon} \frac{dt}{t}<\infty\quad\text{for all }\epsilon>0,\end{equation} but for which the associated ``quasiplane" $f(\RR^{n-1})$ is not $(n-1)$-rectifiable and has locally infinite $\surf$ measure.

To prove Theorem \ref{ThmB} we first prove a version where the maximal dilatation in the Dini condition is replaced with the weak quasisymmetry constant. Recall that a topological embedding $f:\Omega\rightarrow \RR^n$ is called \emph{quasisymmetric} if there exists a homeomorphism $\eta:[0,\infty)\rightarrow[0,\infty)$ such that \begin{equation} \label{etadefn} |x-y|\leq t|x-z|\Rightarrow |f(x)-f(y)|\leq \eta(t)|f(x)-f(z)|\quad\text{for all } x,y,z\in\Omega.\end{equation} Every $K$-quasiconformal map $f:\RR^n\rightarrow\RR^n$ is quasisymmetric for some gauge $\eta_{n,K}$ determined by $n$ and $K$; e.g., see Heinonen \cite{H}. Below we only use (\ref{etadefn}) with $t=1$. This leads to the concept of weak quasisymmetry.

Let $1\leq H<\infty$. An embedding $f:\Omega\rightarrow \RR^n$ is \emph{weakly $H$-quasisymmetric} if \begin{equation}\label{Hdefn} H_f(\Omega) = \sup\left\{\frac{|f(x)-f(y)|}{|f(x)-f(z)|}:x,y,z\in \Omega \text{ and } \frac{|x-y|}{|x-z|}\leq 1\right\}\leq H.\end{equation} We call $H_f(\Omega)$ the \emph{weak quasisymmetry constant} of $f$ on $\Omega$. Also set \begin{equation}\tH_f(\Omega)=H_f(\Omega)-1.\end{equation} We will establish the following theorem in \S2.

\begin{theorem}\label{ThmA} If $f:\RR^n\rightarrow\RR^n$ is quasiconformal and \begin{equation}\label{Hfinite} \int_0^1 \sup_{z\in S^{n-1}} \left(\tH_f(B(z,t))\right)^{2}\frac{dt}{t}<\infty,\end{equation} then the quasisphere $f(S^{n-1})$ admits local $(1+\delta)$-bi-Lipschitz parameterizations, for every $\delta>0$. Thus $f(S^{n-1})$ is $(n-1)$-rectifiable and $\surf(f(S^{n-1}))<\infty$. \end{theorem}

The proof of Theorem \ref{ThmA} is based on the connection between the quasisymmetry of $f$ near $S^{n-1}$ and the flatness of the set $f(S^{n-1})$ (first described in Prause \cite{P}), and a criterion for existence of local bi-Lipschitz parameterizations from Toro \cite{T}.

The maximal dilatation and weak quasisymmetry constant are related as follows. For any $\Omega\subset\RR^n$ and quasiconformal map $f:\Omega\rightarrow\RR^n$, \begin{equation}\label{KHn} K_f(\Omega) \leq \esssup_{x\in\Omega}\frac{\|f'(x)\|^{n-1}}{\ell(f'(x))^{n-1}}\leq H_f(\Omega)^{n-1}.\end{equation} In particular, $\tK_f(\Omega)\leq C \tH_f(\Omega)$ when $H_f(\Omega)$ is close to 1. Hence \begin{equation}\label{HK} \sup_{z\in S^{n-1}}\tH_f(B(z,t))\rightarrow 0\ \text{ as }\ t\rightarrow 0\quad\Longrightarrow\quad \tK_f(A_t)\rightarrow 0\ \text{ as }\ t\rightarrow 0.\end{equation} The question of whether or not the implication in (\ref{HK}) can be reversed is delicate.
When $n=2$ and $f$ is a quasiconformal map of the plane,  $H_f(\RR^2)\leq \exp(C\tK_f(\RR^2))$ (see Theorem 10.33 in \cite{AVV}). Therefore, \begin{equation}\label{bound1}\tH_f(\RR^2)\leq C\tK_f(\RR^2)\quad\text{when }\tK_f(\RR^2)\ll 1.\end{equation} When $n\geq 3$ and $f$ is a quasiconformal map of space (see Theorem 2.7 in \cite{P}), \begin{equation}\label{bound2}\tH_f(\RR^n)\leq C \tK_f(\RR^n) \log\left(\frac{1}{\tK_f(\RR^n)}\right)\quad\text{when }\tK_f(\RR^n)\ll 1.\end{equation} In order to derive Theorem \ref{ThmB} from Theorem \ref{ThmA}, we need (\ref{bound2}) with $B(z,t)$ in place of $\RR^n$, uniformly for all $z\in S^{n-1}$. Unfortunately, to the best of our knowledge such an estimate does not appear in the literature. Thus, in \S3, we show how to localize (\ref{bound2}). We establish an upper bound on the weak quasisymmetry constant of a global quasiconformal map in neighborhoods with maximal dilatation near 1 (see Theorem \ref{2Klemma}). As a consequence, it follows that (see Corollary 3.2) \begin{equation}\label{KH} \tK_f(A_t)\rightarrow 0\ \text{ as }\ t\rightarrow 0\quad\Longrightarrow\quad \sup_{z\in S^{n-1}}\tH_f(B(z,t))\rightarrow 0\ \text{ as }\ t\rightarrow 0.\end{equation} Thus, combining (\ref{HK}) and (\ref{KH}), we conclude that a quasisphere $f(S^{n-1})$ is asymptotically conformal if and only if $\tH_f(B(z,t))\rightarrow 0$ as $t\rightarrow 0$, uniformly across $z\in S^{n-1}$.

The remainder of this paper is divided into two sections, each aimed at the proof of a Dini condition for rectifiability of $f(S^{n-1})$. First we prove Theorem \ref{ThmA} in \S2. Then we prove Theorem \ref{ThmB} in \S3.

\section{Quasisymmetry and Local Flatness}

The goal of this section is to prove Theorem \ref{ThmA}. Following an idea of Prause \cite{P}, we show that the weak quasisymmetry of $f|_{A_t}$ controls the local flatness of $f(S^{n-1})$ at scales depending on $t$. We then invoke a theorem on the existence of local bi-Lipschitz parameterizations from Toro \cite{T}.

Let $\Sigma\subset\RR^n$ $(n\geq 2$) be a closed set. The \emph{local flatness} $\theta_\Sigma(x,r)$ of $\Sigma$ near $x\in \Sigma$ at scale $r>0$ is defined by \begin{equation}\label{thetaapprox}\theta_\Sigma(x,r)=\frac{1}{r}\min_{L\in G(n,n-1)} \HD[\Sigma\cap B(x,r), (x+L)\cap B(x,r)],\end{equation} where $G(n,n-1)$ denotes the collection of $(n-1)$-dimensional subspaces of $\RR^n$ (i.e.\ hyperplanes through the origin) and $\HD[A,B]$ denotes the Hausdorff distance between nonempty, bounded subsets $A,B\subset\RR^n$, \begin{equation}\HD[A,B]=\max\left\{\sup_{x\in A}\,\dist(x,B),\ \sup_{y\in B}\, \dist(y,A)\right\}.\end{equation} Thus local flatness is a gauge of how well a set can be approximated by a hyperplane. Notice that $\theta_{\Sigma}(x,r)$ measures the distance of points in the set to a plane \emph{and} the distance of points in a plane to the set. (By comparison the Jones $\beta$-numbers \cite{J} and Mattila and Vuorinen's linear approximation property \cite{MV} only measure the distance of points in the set to a plane.) Because $\theta_\Sigma(x,r)\leq 1$ for every closed set $\Sigma$, every location $x\in\Sigma$ and every scale $r>0$, this quantity only carries information when $\theta_\Sigma(x,r)$ is small.

Sets which are uniformly close to a hyperplane at all locations and scales first appeared in Reifenberg's solution of Plateau's problem in arbitrary codimension \cite{Reif}. A closed set $\Sigma\subset\RR^n$ is called \emph{$(\delta,R)$-Reifenberg flat} provided that $\theta_{\Sigma}(x,r)\leq \delta$ for all $x\in\Sigma$ and $0<r\leq R$. Moreover, $\Sigma$ is said to be \emph{Reifenberg flat with vanishing constant} if for every $\delta>0$ there exists a scale $R_\delta>0$ such that $\Sigma$ is $(\delta,R_\delta)$-Reifenberg flat. We now record a rectifiability criterion for locally flat sets, which we need below for the proof of Theorem \ref{ThmA}. For further information about flat sets and parameterizations, see the recent investigation by David and Toro \cite{DT2}.

\begin{theorem}[Toro \cite{T}]\label{torotheorem} Let $n\geq 2$. There exists constants $\delta_0>0$, $\epsilon_0>0$ and $C>1$ depending only on $n$ with the following property. Assume that $0<\delta\leq \delta_0$, $0<\epsilon\leq \epsilon_0$, and $\Sigma\subset\RR^n$ is a $(\delta,R)$-Reifenberg flat set. If $x_0\in \Sigma$, $0<r\leq R$, and \begin{equation} \int_{0}^r\sup_{x\in\Sigma\cap B(x_0,r)}\left(\theta_\Sigma(x,t)\right)^2\frac{dt}{t}<\epsilon^2\end{equation} then there exists a bi-Lipschitz homeomorphism $\tau:\Omega\rightarrow \Sigma\cap B(x,r/64)$ where $\Omega$ is a domain in $\RR^{n-1}$; moreover, $\tau$ and $\tau^{-1}$ have Lipschitz constants at most $1+C(\delta+\epsilon)$.\end{theorem}

\begin{corollary}\label{torocor} If $\Sigma\subset\RR^n$ is Reifenberg flat with vanishing constant and \begin{equation}\int_{0}^{t_0} \sup_{x\in\Sigma} \left(\theta_{\Sigma}(x,t)\right)^2\frac{dt}{t}<\infty\quad\text{for some }t_0>0,\end{equation} then $\Sigma$ admits local $(1+\delta)$-bi-Lipschitz parameterizations, for every $\delta>0$. \end{corollary}

The following lemma is based on an observation by Prause \cite{P}, who showed that quasisymmetry bounds the Jones $\beta$-numbers of $f(L)$, where $f$ is quasiconformal and $L$ is a hyperplane. Here we obtain a slightly stronger statement, because we bound the ``two-sided" Hausdorff distance between a set and a hyperplane.

\begin{lemma} \label{flatlemma} Let $U\subset\RR^n$ be an open set which contains the closed ball $\overline{B(0,1)}$. Assume that $f:U\rightarrow \RR^n$ is weakly $H$-quasisymmetric with \begin{equation}\label{apple1} H=1+\epsilon,\quad 0\leq \epsilon \leq 1/20\end{equation} and \begin{equation}\label{apple2} f(\pm e_1)=\pm e_1.\end{equation} If $L=e_1^\perp$ denotes the hyperplane through the origin, orthogonal to the direction $e_1$, then $f(B(0,1))\supset B(f(0),5/6)$ and $\theta_{f(L)}(f(0), 1/2)\leq 20\epsilon$.\end{lemma}

\begin{proof} Assume that $U\subset\RR^n$ is an open set containing the closed ball $\overline{B_1}$, where $B_1=B(0,1)$. Moreover, assume that $f:U\rightarrow\RR^n$ is a weakly $H$-quasisymmetric map satisfying (\ref{apple1}) and (\ref{apple2}). Let $e_1=(1,0,\dots,0)$ and let $L=e_1^\perp$ denote the $(n-1)$-dimensional plane through the origin that is orthogonal to the direction $e_1$. By the polarization identity, for any $y\in\RR^n$, \begin{equation}\label{apple3} \dist(y,L)=|\langle y,e_1\rangle|=\frac{1}{4}\left||y-e_1|^2-|y+e_1|^2\right|.\end{equation} Let $x\in L\cap B_1$. Since $|x-e_1|=|x+e_1|$ and $f(\pm e_1)=\pm e_1$, by  weak quasisymmetry, \begin{equation}\label{apple4}\frac{1}{1+\epsilon}\leq \frac{|f(x)-e_1|}{|f(x)+e_1|}\leq 1+\epsilon.\end{equation} Combining (\ref{apple3}) and (\ref{apple4}), \begin{equation}\label{apple5}\dist(f(x),L)\leq \frac{1}{4}\left((1+\epsilon)^2-1\right)\max\{|f(x)\pm e_1|^2\}.\end{equation} On the other hand, since $|x\pm e_1|\leq 2=|\pm e_1-(\mp e_1)|$ for every $x\in B_1$, the weak quasisymmetry of $f$ yields \begin{equation}\label{apple6} |f(x)\pm e_1|\leq 2(1+\epsilon).\end{equation} Thus, from (\ref{apple1}), (\ref{apple5}) and (\ref{apple6}), we conclude \begin{equation}\label{apple7}\dist(f(x),L)\leq \left((1+\epsilon)^2-1\right)(1+\epsilon)^2=\epsilon(2+\epsilon)(1+\epsilon)^2\leq 2.5\epsilon.\end{equation} Since the hyperplane $f(0)+L=\langle f(0),e_1\rangle e_1 +L=(\pm \dist(f(0),L))e_1+L$, it follows that \begin{equation}\label{apple8}\dist(f(x),f(0)+L)\leq 5\epsilon\quad\text{for all }x\in L\cap B_1.\end{equation} So far we have bounded the distance of points in the set $f(L\cap B_1)$ to the hyperplane $f(0)+L$. To estimate the local flatness of $f(L)$ near $f(0)$, we also need to bound the distance of points in $f(0)+L$ to the set $f(L)$.

First we claim that $f(B_1)\supset B(f(0),5/6)$. To verify this, suppose that $z\in\RR^n$, $|z|=1$. On one hand,
by weak quasisymmetry, (\ref{apple1}) and (\ref{apple2}), \begin{equation}\label{apple9} |f(z)-f(0)|\geq \frac{1}{1+\epsilon} |f(e_1)-f(0)|\geq \frac{20}{21}|e_1-f(0)|.\end{equation} On the other hand, pick $w\in L$ such that $|f(0)-w|=\dist(f(0),L)$. Then (\ref{apple1}), (\ref{apple7}) and the triangle inequality yield \begin{equation}\label{apple10}|e_1-f(0)|\geq |e_1-w|-|f(0)-w|\geq 1-2.5\epsilon\geq 7/8\end{equation} Together (\ref{apple9}) and (\ref{apple10}) imply $|f(z)-f(0)|\geq 5/6$. Because $|f(z)-f(0)|\geq 5/6$ for all $|z|=1$ and $f$ is homeomorphism, we conclude that $f(B_1)\supset B(f(0),5/6)$. Hence, by (\ref{apple8}), $f(L)\cap B(f(0),5/6)$ is contained in a $5\epsilon$-neighborhood of $f(0)+L$.

Next we consider the sets $S^\pm=\{x\in B(f(0),5/6): \langle x-f(0),\pm e_1\rangle >5\epsilon\}$. Because the hyperplane $L$ divides $B_1$ into two connected components and the map $f$ is a homeomorphism, $f(L)$ divides $f(B_1)$ into two connected components. Hence, in view of (\ref{apple8}), we know that $S^+$ and $S^-$ are contained in different connected components of $B(f(0),5/6)\setminus f(L)$. In particular, every line segment from $\partial S^+$ to $\partial S^-$ intersects $f(L)$. If $v\in B(f(0),1/2)\cap (f(0)+L)$, then $v\pm 5\epsilon e_1\in B(f(0),5/6)$ (since $(1/2)^2+(5\epsilon)^2<5/6$). Thus the line segment $\ell_v$ with endpoints $v\pm 5\epsilon e_1$ in $\partial S^\pm$ necessarily intersects $f(L)$. Since $v$ is the center of $\ell_v$, it follows that \begin{equation}\label{apple11} \dist(v,f(L))\leq 5\epsilon\quad\text{ for all }v\in B(f(0),1/2))\cap (f(0)+L).\end{equation} Moreover, if $v\in B(f(0),1/2-5\epsilon)\cap(f(0)+L)$, then $\ell_v\subset B(f(0), 1/2)$ and the upper bound in (\ref{apple11}) remains valid with $\dist(v,f(L)\cap B(f(0),1/2))$ in place of $\dist(v,f(L))$. On the other hand, since $f(0)+L$ is a hyperplane, for any $$v\in B(f(0),1/2)\cap (f(0)+L)\setminus B(f(0),1/2-5\epsilon)$$ there exists $v'\in B(f(0),1/2-5\epsilon)\cap (f(0)+L)$ such that $|v-v'|\leq 5\epsilon$. Therefore, \begin{equation}\label{apple12} \dist(v,f(L)\cap B(f(0),1/2))\leq 10\epsilon\quad\text{ for all }v\in B(f(0),1/2)\cap (f(0)+L).\end{equation}

To finish, we note that taking the maximum of (\ref{apple8}) and (\ref{apple12}) yields \begin{equation}\HD[f(L)\cap B(f(0),1/2), (f(0)+L)\cap B(f(0),1/2)]\leq 10\epsilon.\end{equation} It immediately follows that $\theta_{f(L)}(f(0),1/2)\leq 20\epsilon$.
\end{proof}

\begin{figure}\includegraphics[width=.58\textwidth]{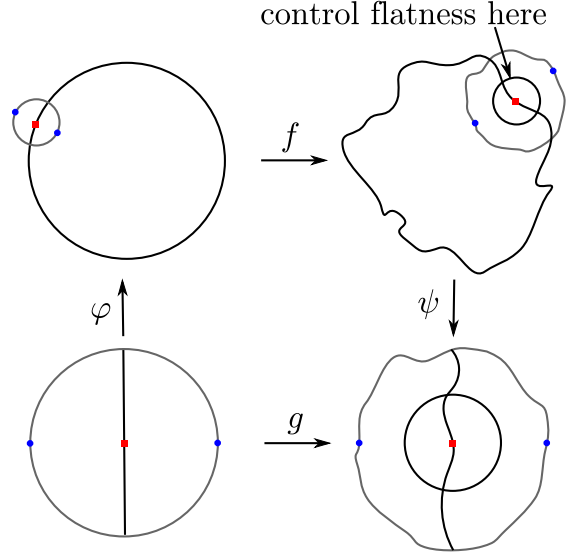}\caption{The relationship between $f$ and $g$}\end{figure}

Let us now illustrate how to use Lemma \ref{flatlemma} with quasispheres. Fix $\epsilon\in[0,1/20]$. Suppose that $f:\RR^n\rightarrow\RR^n$ is a quasiconformal map and suppose that there exists $r>0$ such that $\tH_f(B(z,2r))\leq \epsilon$ for all $z\in S^{n-1}$. Pick $z\in S^{n-1}$ (the red square point in the upper left of Figure 1), let $\vec n_z$ denote a unit normal vector to $S^{n-1}$ at $z$, and assign $e^r_{\pm}=z\pm r\vec n_z$ (the blue round points). Choose any affine transformation $\varphi:\RR^n\rightarrow\RR^n$, which is a composition of a translation, a rotation and a dilation and such that $\varphi(0)=z$ and $\varphi(\pm e_1)=e^r_{\pm}$. Then choose another affine transformation $\psi:\RR^n\rightarrow\RR^n$, which is a composition of a translation, a rotation and a dilation and such that $\psi((f(e_+^r)+f(e_-^r))/2)=0$ and $\psi(f(e^r_\pm))=\pm e_1$. Finally put $g=\psi\circ f\circ\varphi$. Then $g|_{B(0,2)}$ is weakly $H_f(B(z,2r))$-quasisymmetric and $g(\pm e_1)=\pm e_1$.
By Lemma 2.3, $\theta_{g(e_1^{\perp})}(g(0),1/2)\leq 20 \tH_f(B(z,2r))$. Hence, letting $L_z$ denote the tangent plane to $S^{n-1}$ at $z$, this is equivalent to $\theta_{\psi\circ f(L_z)}(\psi(f(z)),1/2)\leq 20 \tH_f(B(z,2r))$. Since the dilation factor of $\psi^{-1}$ is $|f(e^r_+)-f(e^r_-)|/2$, it follows that
\begin{equation}\label{orange1} \theta_{f(L_z)}\left(f(z),\frac{1}{4}|f(e^r_+)-f(e^r_-)|\right)\leq 20 \tH_f(B(z,2r)).\end{equation} Since Lemma \ref{flatlemma} also implies that $g(B(0,1))\supset B(g(0),5/6)$, we similarly get that \begin{equation}\label{orange1b} f(B(z,r))\supset B\left(f(z),\frac{5}{12}|f(e^r_+)-f(e^r_-)|\right).\end{equation}  Notice that since $\tH_f(B(z,2s))\leq \tH_f(B(z,2r))$ for all $0<s\leq r$, a similar argument shows that (\ref{orange1}) and (\ref{orange1b}) hold with $s$ in place of $r$ for all $0<s\leq r$.

We now apply the local H\"{o}lder continuity of the quasiconformal map $f$.  Since $K_f(B(z,r))\leq H_f(B(z,r))^{n-1}\leq (1+\epsilon)^{n-1}$, there is a constant $C>1$ depending only on $n$ (because $0\leq \epsilon\leq 1/20$, e.g.~ see Theorem 11.14 in \cite{Vu-book} for a precise version of the local H\"{o}lder continuity we use here)
\begin{equation}\label{orange2a} |f(x)-f(y)|\leq CM_{z,r}|x-y|^{\alpha}\quad\text{for all } x,y\in B(z,r/2)\end{equation} where $\alpha=(1+\epsilon)^{(n-1)/(1-n)}=(1+\epsilon)^{-1}$ and where $M_{z,r}=\sup_{|x-z|=r}|f(x)-f(z)|$. We also want to find a lower bound on $|f(x)-f(y)|$. First by two applications of weak quasisymmetry \begin{equation}\label{orange2b} |f(e_+^r)-f(e_-^r)|\geq \frac{1}{1+\epsilon} |f(e_+^r)-f(z)|\geq \left(\frac{1}{1+\epsilon}\right)^2M_{z,r}.\end{equation} Combining (\ref{orange1b}) and (\ref{orange2b}), we conclude $f(B(z,r))$ contains $B(f(z),M_{z,r}/3)$. Thus, $K_{f^{-1}}(B(f(z),M_{z,r}/3))\leq K_f(B(z,r))\leq (1+\epsilon)^{n-1}$ and analogously to (\ref{orange2a}) we have that \begin{equation}\label{orange2c} |f^{-1}(\xi)-f^{-1}(\eta)|\leq Cr|\xi-\eta|^{\alpha}\quad\text{for all } \xi,\eta\in B(f(z),M_{z,r}/6).\end{equation} Suppose that we specified $r\leq r_0:=\min_{0\leq \epsilon\leq 1/20}2(1/6C)^{1/\alpha}$, which depends only on $n$. Then $|f(x)-f(z)|\leq M_{z,r}/6$ for all $x\in B(z,r/2)$ by (\ref{orange2a}). Thus, we can apply (\ref{orange2c}) with $x,y\in B(z,r/2)$, $\xi=f(x)$ and $\eta=f(y)$ to get that \begin{equation}\label{orange2d} |x-y| \leq Cr_0|f(x)-f(y)|^\alpha\quad\text{for all }x,y\in B(z,r/2).\end{equation} In particular, for all $0<s\leq r/2$, \begin{equation} t=\frac14|f(e_+^s)-f(e_-^s)| \geq \frac{1}{4}\left(\frac{2s}{Cr_0}\right)^{1/\alpha}=:c(2s)^{1/\alpha}.\end{equation}
Thus, since $2s\leq (t/c)^{\alpha}$ and $\tH_f(B(z,\rho))$ is increasing in $\rho$, we conclude that \begin{equation}\label{orange3} \theta_{f(L_z)}(f(z),t)\leq 20\tH_f(B(z,(t/c)^\alpha))\end{equation} for all $t>0$ such that $(t/c)^{\alpha}\leq r/2$. We remark that by replacing $c$ by $\min_{0\leq \epsilon\leq 1/20} c$, we may assert that $c$ depends only on $n$.

We now want to transfer the estimate (\ref{orange3}) for the local flatness of the image $f(L_z)$ of the tangent plane $L_z$ to an estimate for the quasicircle $f(S^{n-1})$. Evidently \begin{equation}\label{orange4}\begin{split} \theta_{f(S^{n-1})} (f(z),t)\leq &\ 20\tH_f(B(z,(t/c)^\alpha))\\&+\frac{1}{t}\HD[f(S^{n-1})\cap B(f(z),t), f(L_z)\cap B(f(z),t)].\end{split}\end{equation}
Thus our next task is to estimate $\HD[f(S^{n-1})\cap B(f(z),t), f(L_z)\cap B(f(z),t)]$. First we note that by elementary geometry there is an absolute constant $C_0$ so that \begin{equation} \HD[S^{n-1}\cap B(z,\rho), L_z\cap B(z,\rho)]\leq C_0\rho^2\quad\text{for all }0<\rho\leq 1/2.\end{equation}
Thus, by the local H\"older continuity (\ref{orange2a}) and (\ref{orange2d}) and the constraint $0\leq\epsilon\leq 1/20$, there exist constants $c_1>0$ and $C_1>1$ that depend only on $n$ and $M_{z,r}$ so that \begin{equation}\label{orange5} \HD[f(S^{n-1})\cap B(f(z),c_1\rho^{1/\alpha}),f(L_z)\cap B(f(z),c_1\rho^{1/\alpha})]\leq C_1 \rho^{2\alpha}\end{equation} for all $\rho>0$ such that $c_1\rho^{1/\alpha}\leq r/2$. With $t=c_1\rho^{1/\alpha}$, the estimate (\ref{orange5}) becomes \begin{equation} \label{orange6} \HD[f(S^{n-1})\cap B(f(z),t),f(L_z)\cap B(f(z),t)]\leq C_2 t^{2\alpha^2}\end{equation} where $C_2>1$ depends on $n$ and $M_{z,r}$.
Substituting (\ref{orange6}) into (\ref{orange4}), we get that, for all $t>0$ sufficiently small, \begin{equation} \theta_{f(S^{n-1})}(f(z),t)\leq 20\tH_f(B(z,(t/c)^\alpha))+C_2t^{2\alpha^2-1}.\end{equation} Observe that $\beta:=2\alpha^2-1\in(0,\alpha]$ since $0\leq \epsilon\leq 1/20$. Therefore, for all $t>0$ sufficiently small, $\tH_f(B(z,(t/c)^{\alpha})) \leq \tH_f(B(z,(t/c)^\beta))$ and \begin{equation} \theta_{f(S^{n-1})}(f(z),t)\leq 20\tH_f(B(z,(t/c)^\beta))+C_2t^{\beta}.\end{equation} We have outlined the proof of the following theorem.

\begin{theorem}\label{flatwithepsilon} Let $f:\RR^n\rightarrow\RR^n$ be a quasiconformal map. If there exists $r>0$ such that $\tH_f(B(z,2r))\leq \epsilon\leq 1/20$ for all $z\in S^{n-1}$, then there exist constants $c>0$ and $C>1$ depending only on $n$ and $M_r=\sup_{z\in S^{n-1}}\sup_{|x-z|=r}|f(x)-f(z)|$ and a constant $t_0>0$ depending only on $n$, $M_{r}$ and $r$ such that \begin{equation} \label{flatestimate} \sup_{z\in S^{n-1}}\theta_{f(S^{n-1})}(f(z),t)\leq 20\sup_{z\in S^{n-1}}\tH_f(B(z,(t/c)^{\beta}))+Ct^{\beta}\end{equation} for all $0<t\leq t_0$, where $\beta=2\alpha^2-1$, $\alpha=(1+\epsilon)^{-1}$.\end{theorem}

Note that $\beta\uparrow 1$ as $\epsilon\downarrow 0$. Theorem \ref{flatwithepsilon} has several immediate consequences.

\begin{corollary}\label{flattheorem} If $f:\RR^n\rightarrow\RR^n$ is quasiconformal and $\sup_{z\in S^{n-1}}\tH_f(B(z,r))\rightarrow 0$ as $r\rightarrow 0$, then for all $0<\beta<1$ there exist constants $c,t_0>0$ and $C>1$ depending on $f$ and $\beta$ such that (\ref{flatestimate})  holds for all $0<t\leq t_0$.\end{corollary}

\begin{corollary}\label{reifflatthm} Let $0<\delta<1$. If $f:\RR^n\rightarrow\RR^n$ is a quasiconformal map and $\sup_{z\in S^{n-1}}\tH_f(z,r)\leq \delta/40$ for some $r>0$, then $f(S^{n-1})$ is $(\delta,R)$-Reifenberg flat for some $R>0$.\end{corollary}

\begin{corollary}\label{reifcor} If $f:\RR^n\rightarrow\RR^n$ is quasiconformal and $\sup_{z\in S^{n-1}}\tH_f(B(z,r))\rightarrow 0$ as $r\rightarrow 0$ , then $f(S^{n-1})$ is Reifenberg flat with vanishing constant. \end{corollary}

\begin{remark}\label{asympdim} Mattila and Vuorinen \cite{MV} demonstrated that sets with the $(\delta,R)$-linear approximation property (see \cite{MV} for the definition) have Hausdorff dimension at most $n-1+C\delta^2$, $C=C(n)>1$. Since $(\delta,R)$-Reifenberg flat sets also have the $(\delta,R)$-linear approximation property, Mattila and Vuorinen's theorem and Corollary \ref{reifflatthm} imply the following bound. If $f:\RR^n\rightarrow\RR^n$ is quasiconformal, then \begin{equation}\label{dimbd} \dim_H f(S^{n-1}) \leq n-1+C\inf_{r>0}\sup_{z\in S^{n-1}} \left(\tH_f(B(z,r))\right)^2.\end{equation} On the other hand, every quasisphere has Hausdorff dimension at least $n-1$. Therefore, (\ref{KH}) and (\ref{dimbd}) imply that every asymptotically conformal quasisphere $f(S^{n-1})$ has Hausdorff dimension $n-1$.\end{remark}

We can now use Corollaries \ref{torocor}, \ref{flattheorem} and \ref{reifcor} to prove Theorem \ref{ThmA}.

\begin{proof}[Proof of Theorem \ref{ThmA}] Assume that $f:\RR^n\rightarrow\RR^n$ is a quasiconformal mapping such that (\ref{Hfinite}) holds. Since the function $\sup_{z\in S^{n-1}} \tH_f(B(z,\cdot))$ is decreasing, (\ref{Hfinite}) implies that $\sup_{z\in S^{n-1}}\tH_f(B(z,t))\rightarrow 0$ as $t\rightarrow 0$. Hence $f(S^{n-1})$ is Reifenberg flat with vanishing constant by Corollary \ref{reifcor}. Pick any $\beta\in[1/2,1)$ and let $t_0, c, C>0$ be the constants from Corollary \ref{flattheorem} such that (\ref{flatestimate}) holds. Since $(a+b)^2\leq 2a^2+2b^2$, \begin{equation}\begin{split}&\int_0^{t_0} \sup_{z\in S^{n-1}}\left(\theta_{f(S^{n-1})}(f(z), t)\right)^2 \frac{dt}{t} \\ &\qquad\leq
 800 \int_0^{t_0} \sup_{z\in S^{n-1}}\left(\tH_f(B(z,(t/c)^\beta))\right)^2 \frac{dt}{t}
+ 2C^2\int_0^{t_0} t^{2\beta}\frac{dt}{t}.\end{split}\end{equation} On one hand, using the change of variables $s=(t/c)^\beta$, $ds/s = (\beta/c)(dt/t)$, \begin{equation}\begin{split} &\int_{0}^{t_0} \sup_{z\in S^{n-1}}\left(\tH_f(B(z,(t/c)^\beta))\right)^2 \frac{dt}{t}\\ &\qquad=\frac{c}{\beta}\int_0^{(t_0/c)^\beta}\sup_{z\in S^{n-1}}\left(\tH_f(B(z,s))\right)^2\frac{ds}{s}<\infty\end{split}\end{equation} by (\ref{Hfinite}). On the other hand, since $\beta\geq 1/2$, $\int_0^{t_0} t^{2\beta-1}dt<\infty$. Therefore, \begin{equation}\int_0^{t_0}\sup_{z\in S^{n-1}}\left(\theta_{f(S^{n-1})}(f(z),t)\right)^2\frac{dt}{t}<\infty,\end{equation} and the quasisphere $f(S^{n-1})$ admits local $(1+\delta)$-bi-Lipschitz parameterizations for every $\delta>0$ by Corollary \ref{torocor}. It follows that $f(S^{n-1})$ is $(n-1)$-rectifiable. Moreover, since $f(S^{n-1})$ is compact, we conclude $\surf(f(S^{n-1}))<\infty$. \end{proof}

\section{Local Bounds on Quasisymmetry}

In this section, our goal is to derive Theorem \ref{ThmB} from Theorem \ref{ThmA}. However, we face a technical challenge. We need to use the maximal dilatation of $f$ near $S^{n-1}$ to bound the weak quasisymmetry constant of $f$ near $S^{n-1}$. Our solution to this puzzle is Theorem \ref{2Klemma}.

\begin{theorem} \label{2Klemma} Given $n\geq 2$ and $1<K\leq \min\{4/3,K'\}$, set \begin{equation}\label{RR} R= \left(\frac{c}{K-1}\right)^{c/(K-1)}\end{equation} where $c>1$ is a constant that only depends on $n$ and $K'$. Assume that $f:\RR^n\rightarrow\RR^n$ is $K'$-quasiconformal. If, in addition, $K_f(B(w,Rs))\leq K$, then
\begin{equation}\tH_f(B(w,s))\leq C(K-1)\log\left(\frac{1}{K-1}\right)\end{equation}
where $C>1$ is an absolute constant.\end{theorem}

\begin{corollary}\label{Hasymp} If $f:\RR^n\rightarrow\RR^n$ is quasiconformal and $\tK_f(A_t)\rightarrow 0$ as $t\rightarrow 0$, then $\sup_{z\in S^{n-1}}\tH_f(B(z,t))\rightarrow 0$ as $t\rightarrow 0$.\end{corollary}

Before we prove Theorem \ref{2Klemma}, let's use it to establish Theorem \ref{ThmB}.

\begin{proof}[Proof of Theorem \ref{ThmB}] Suppose that $f:\RR^n\rightarrow\RR^n$ is a $K'$-quasiconformal map satisfying (\ref{Kfinite}). We need to verify that $f$ also satisfies the Dini condition (\ref{Hfinite}). Fix $z\in S^{n-1}$. In order to apply Theorem \ref{2Klemma} and perform a change of variables at a certain step below, we choose a majorant $M(t)$ of $\tK_f(A_t)$ on $(0,r_0)$, as follows. If $\tK_f(A_r)=0$ for some $r>0$, then set $r_0=r$ and let $M(t)=t$ for all $0<t<r_0$. Otherwise, let $r_0=1$ and let $M(t)$ any smooth increasing function such that \begin{equation} \tK_f(A_t)\leq M(t)\leq b\tK_f(A_t)\quad\text{for all }0<t<1\end{equation} for some constant $b>1$.  (We note that $M(t)$ and $b$ exist by standard techniques. Also the constant $b$ depends only on $\tK_f(A_1)$ and hence is fully determined by $K'$.) In both cases, $M(t)$ is a smooth increasing function such that $\lim_{t\rightarrow 0+} M(t)=0$, such that $\sup_{z\in S^{n-1}}\tK_f(B(z,t))\leq \tK_f(A_t)\leq M(t)$ for all $0<t<r_0$ and such that \begin{equation}\label{Mdini} \int_0^{r_0} \left(M(t)\log\frac{1}{M(t)}\right)^2\frac{dt}{t}<\infty.\end{equation} By Theorem \ref{2Klemma}, there exist constants $c,C>1$ which depend only on $n$ and $K'$ such that for all $t>0$ satisfying $M(t)\leq\min\{1/3, K'-1\}$, \begin{equation}\label{golf1}
\sup_{z\in S^{n-1}}\tH_f (B(z,\phi(t))) \leq C M(t)\log \left( \frac{1}{M(t)} \right) \end{equation} where
\begin{equation}\label{golf2} \phi(t) = t\left( \frac{M(t)}{c} \right)^{c/M(t)}.\end{equation} Since $M(t)\rightarrow 0$ as $t\rightarrow 0$, there is $t_0\in (0,r_0)$ such that (\ref{golf1}) holds for all $0<t\leq t_0$. Write $\psi(t)=M(t)/c$, so that $\phi(t)=t\exp\left(\frac{1}{\psi(t)}\log\psi(t)\right)$. Then \begin{equation}\label{golf5a} \phi'(t) = \exp\left(\frac{1}{\psi(t)}\log\psi(t)\right)\left[1+t\frac{\psi'(t)}{\psi(t)^2}\left(1+\log\frac{1}{\psi(t)}\right)\right] \end{equation} and \begin{equation}\label{golf5} \frac{\phi'(t)}{\phi(t)} =\frac{1}{t}+\frac{\psi'(t)}{\psi(t)^2}\left(1+\log\frac{1}{\psi(t)}\right).\end{equation} Notice that since $M(t)$ is increasing and $M(t)\rightarrow 0$ as $t\rightarrow 0$, we have $\phi'(t)>0$ for all $t$ sufficiently small. Thus, we can apply a change of variables $s=\phi(t)$, $ds/s=(\phi'(t)/\phi(t))dt$ to obtain
\begin{equation}\begin{split}\label{golf3} \int_0^{\phi(t_0)} \sup_{z\in S^{n-1}}\left(\tH_f(B(z,s))\right)^2\frac{ds}{s} &= \int_{0}^{t_0} \sup_{z\in S^{n-1}}\left(\tH_f(B(z,\phi(t)))\right)^2 \frac{\phi'(t)}{\phi(t)}dt\\ &\leq C^2\int_0^{t_0} \left( M(t) \log \frac{1}{M(t)} \right)^2 \frac{\phi'(t)}{\phi(t)} dt. \end{split} \end{equation}  To establish (\ref{Hfinite}), it remains to show that the integral on the right hand side of (\ref{golf3}) is finite. Using (\ref{golf5}) the integral on the right hand side of (\ref{golf3}) is equal to \begin{equation}\label{golf6} \int_0^{t_0} \left( M(t) \log \frac{1}{M(t)} \right)^2\frac{dt}{t} +c^2\int_0^{t_0}\left(\log\frac{1}{c\psi(t)}\right)^2\left(1+\log\frac{1}{\psi(t)}\right)\psi'(t)dt.\end{equation} The first term on the right hand side of (\ref{golf6}) is finite by (\ref{Mdini}). And, after changing variables again, the second term on the right hand side of (\ref{golf6}) becomes  \begin{equation}\label{golf7} c^2\int_0^{\psi(t_0)}\left(\log\frac{1}{ct}\right)^2\left(1+
\log\frac{1}{t}\right)dt,\end{equation} which is finite too.  This verifies (\ref{Hfinite}). Therefore, by Theorem \ref{ThmA}, the quasisphere $f(S^{n-1})$ admits local $(1+\delta)$-bi-Lipschitz parameterizations for every $\delta>0$.\end{proof}

It remains to prove Theorem \ref{2Klemma}. Rather than estimate the weak quasisymmetry constant $H_f(B(w,s))$ directly, we shall instead estimate a related extremal problem for standardized quasiconformal maps.

\begin{definition} \ 
\begin{enumerate}
\item[(i)] We say that a quasiconformal map $g:\RR^n\rightarrow\RR^n$ is \emph{standardized} if $g(0)=0$, $g(e_1)=e_1$, and $g(B(0,1))\subset B(0,1)$.
\item[(ii)] Let $f:\RR^n\rightarrow\RR^n$ be any quasiconformal map, let $x\in \RR^n$ and let $r>0$. A quasiconformal map $g:\RR^n\rightarrow\RR^n$ is a \emph{standardization} of $f$ with respect to $B(x,r)$ if there exist affine transformations $\phi$ and $\psi$ of $\RR^n$ such that $\phi(B(0,1))=B(x,r)$ and $g=\psi\circ f\circ \phi$ is standardized.
\end{enumerate} \end{definition}

\begin{remark}[How to Estimate $H_f(B(w,s)$] \label{howto} Let $f:\RR^n\rightarrow\RR^n$ be $K'$-quasiconformal, let $R>1$ and assume that $K_f(B(w,Rs))\leq K$. Suppose that we want to estimate $|f(\tilde x)-f(\tilde z)|/|f(\tilde y)-f(\tilde z)|$ for some $\tilde x,\tilde y,\tilde z\in B(w,s)$ with $|\tilde x-\tilde z|\leq |\tilde y-\tilde z|$. Then, writing $r:=|\tilde y-\tilde z|$, $B(\tilde z,r)\subset B(w,3s)$, $K_f(B(\tilde z,\frac{R}{3}r))\leq K$ and \begin{equation}\begin{split}\label{pecan1} \frac{|f(\tilde x)-f(\tilde z)|}{|f(\tilde y)-f(\tilde z)|} &\leq \max\left\{\frac{|f(x)-f(\tilde z)|}{|f(y)-f(\tilde z)|}:|x-\tilde z|\leq |y-\tilde z|= r \right\}\\
&= \max\left\{\frac{|f(x)-f(\tilde z)|}{|f(y)-f(\tilde z)|}:|x-\tilde z|=|y-\tilde z|=r\right\}.\end{split}\end{equation} Suppose that the maximum in (\ref{pecan1}) is obtained at $x=x^\star$ and $y=y^\star$. This implies that $|f(x^\star)-f(\tilde z)|\geq |f(y)-f(\tilde z)|$ for all $y$ such that $|y-\tilde z|=r$. Let $\phi$ be any affine transformation of $\RR^n$ sending $B(0,1)$ to $B(\tilde z,r)$ with $\phi(0)=\tilde z$ and $\phi(e_1)=x^\star$, and let $\psi$ be any affine transformation  of $\RR^n$ sending $B(f(\tilde z), |f(x^\star)-f(\tilde z)|)$ to $B(0,1)$ with $\psi(f(\tilde z))=0$ and $\psi(f(x^\star))=e_1$. Then $F=\psi\circ f\circ\phi$ is a standardization of $f$ with respect to $B(\tilde z,r)$, $f$ is $K'$-quasiconformal, $K_F(B(0,R/3))\leq K$ and \begin{equation} \frac{|f(\tilde x)-f(\tilde z)|}{|f(\tilde y)-f(\tilde z)|} \leq \frac{|F(e_1)-F(0)|}{|F(\phi^{-1}(y^\star))-F(0)|}=\frac{|F(e_1)|}{|F(\phi^{-1}(y^\star))|}\end{equation} where $|\phi^{-1}(y^\star)|=1$. Let $\mathcal{G}(K',K,R/3)$ denote the collection of all standardized $K'$-quasiconformal maps $g$ such that $K_g(B(0,R/3))\leq K$. Then \begin{equation}\label{pecan2} \frac{|f(\tilde x)-f(\tilde z)|}{|f(\tilde y)-f(\tilde z)|} \leq \max\left\{\frac{|g(x)|}{|g(y)|}:g\in\mathcal{G}(K',K,R/3),\,|x|=|y|=1\right\}.\end{equation}
Therefore, since the right hand side of (\ref{pecan2}) is independent of $\tilde x$, $\tilde y$ and $\tilde z$, \begin{equation}\label{pecan3} H_f(B(w,s))\leq \max\left\{\frac{|g(x)|}{|g(y)|}:g\in\mathcal{G}(K',K,R/3),\,|x|=|y|=1\right\}\end{equation} for every $K'$-quasiconformal map $f$ of $\RR^n$ such that $K_f(B(w,Rs))\leq K$.\end{remark}

The extremal problem described in Remark \ref{howto} (\ref{pecan3}) has been studied in the special case $K=K'$ by several authors; see Vuorinen \cite{Vu}, Seittenranta \cite{Se}, and most recently, Prause \cite{P}. Our idea to prove Theorem \ref{2Klemma} is to modify the method from \cite{Vu} to incorporate two estimates on the maximal dilatation ($K_g(\RR^n)\leq K'$ \emph{and} $K_g(B(0,R/3))\leq K$). To do this we need to work with the geometric definition of quasiconformal maps.  Recall that according to the geometric definition a map $f:\Omega\rightarrow\Omega'$ between domains $\Omega,\Omega'\subset\RR^n$ is \emph{$K$-quasiconformal} ($1\leq K<\infty$) provided that $f$ is a homeomorphism and the inequalities \begin{equation}\label{gmod} K^{-1}\Mod(\Gamma) \leq \Mod(f\Gamma)\leq K\Mod(\Gamma)\end{equation} hold for every curve family $\Gamma$ in $\Omega$. Here $\Mod(\Gamma)$ refers to the $n$-modulus of $\Gamma$; e.g., see Heinonen \cite{H}. We also define the \emph{maximal dilatation} $K_f(\Omega)$ to be the smallest $K$ such that the inequalities (\ref{gmod}) hold for all $\Gamma$. It is well known that the analytic and geometric definitions of quasiconformal maps  coincide; e.g., see Chapter 4, \S36 of V\"ais\"al\"a \cite{V}.

As a preliminary step towards the proof of Theorem \ref{2Klemma}, we record some facts about the modulus of the \emph{Teichm\"uller ring} $[-e_1,0]\cup[se_1,\infty]$ in $\RR^n$ with $s>0$ and the \emph{Gr\"otzsch ring} $\overline{B(0,1)}\cup [te_1,\infty)$ in $\RR^n$ with $t>1$.
For every pair of disjoint sets $E,F\subset\RR^n$, $(E,F)$ is the family of curves connecting $E$ and $F$ in $\RR^n$.

\begin{lemma}\label{teich} With $n\geq 2$ fixed, assign $\tau_n(s)=\Mod([-e_1,0],[se_1,\infty))$ for all $s>0$, and assign $\gamma_n(t)=\Mod(B(0,1),[te_1,\infty))$ for all $t>1$. The functions $\tau_n$ and $\gamma_n$ are decreasing homeomorphisms onto $(0,\infty)$. Moreover,
\begin{equation}\label{grot} \gamma_n(t)=2^{n-1}\tau_n(t^2-1)\quad\text{for all }t>1\end{equation} and \begin{equation}\label{grotbound} \frac{\sigma_{n-1}}{\left(\log \lambda_n t\right)^{n-1}}\leq \gamma_n(t) \leq \frac{\sigma_{n-1}}{\left(\log t\right)^{n-1}}\quad\text{for all $t>1$}\end{equation} where $\sigma_{n-1}=\mathcal{H}^{n-1}(S^{n-1})$ is the surface area of the unit sphere and $\lambda_n\in[4,2e^{n-1})$ is the Gr\"otzsch constant. For all $A>0$, define the distortion function \begin{equation} \label{phiA} \varphi_{A,n}(r)=\frac{1}{\gamma_n^{-1}(A\gamma_n(1/r))}\quad\text{for all $0<r<1$}.\end{equation} Then $\varphi_{n,A}$ is an increasing homeomorphism from $(0,1)$ to $(0,1)$. \end{lemma}

\begin{proof} We refer the reader to \S7 of Vuorinen \cite{Vu-book}.\end{proof}

In the special case $K=K'$, the following calculation appears in slightly different form in Seittenranta \cite{Se} (c.f. Theorem 1.5 and Lemma 3.1 in \cite{Se}) and Prause \cite{P} (c.f. Theorem 2.7 and Theorem 3.1 in \cite{P}).

\begin{lemma}\label{gammabounds} Let $\mathcal{G}_R=\mathcal{G}(K',K,R)$ be the family of standardized $K'$-quasiconformal maps $g:\RR^n\rightarrow\RR^n$ such that $K_g(B(0,R))\leq K$. For all $1<A\leq 16/9$ there exists $t_0\in(0,1)$ (see (\ref{t0})) depending only on $n$ and $A$ such that if the inequalities
\begin{equation}\label{gamma1} \gamma_n\left(\sqrt{1+\frac{1}{|g(x)|}}\right)\leq A\gamma_n\left(\sqrt\frac{1}{|x|}\right)\end{equation}
and
\begin{equation}\label{gamma2}\gamma_n\left(\sqrt{1+\frac{1}{|x|}}\right)\leq A\gamma_n\left(\sqrt\frac{1}{|g(x)|}\right)\end{equation} hold for all $g\in\mathcal{G}_R$ and $|x|=t_0$, then
\begin{equation} \label{plog} \max\left\{\frac{|h(x)|}{|h(y)|}:h\in\mathcal{G}_{R/t_0},\,|x|=|y|=1 \right\}\leq 1 + C(A-1)\log\left(\frac{1}{A-1}\right).\end{equation} for some absolute constant $C>1$.
\end{lemma}

\begin{proof}

Let $n$, $K'$, $K$, $R$ and $A$ be given and fix $t_0\in(0,1)$ to be specified later. Suppose that (\ref{gamma1}) and (\ref{gamma2}) hold for all $g\in\mathcal{G}_{R}$ when $|x|=t_0$. We remark that (\ref{gamma1}) and (\ref{gamma2}) make sense, since $0<|g(x)|<1$ when $0<|x|<1$ because $g$ is standardized. Since $\gamma_n$ is strictly decreasing, one can apply $\gamma^{-1}_n$ to both sides of (\ref{gamma1}), invoke the definition of the distortion function (\ref{phiA}), and perform basic manipulations to get
\begin{equation}\label{pumpkin1} |g(x)| \leq \frac{ \varphi_{A,n}^2(\sqrt{t_0})}{1-\varphi_{A,n}^2(\sqrt{t_0})}=:\mathcal{A}(t_0)\quad\text{for all }g\in\mathcal{G}_R.\end{equation} Similarly, after first dividing (\ref{gamma2}) through by $A$, one can apply $\gamma^{-1}_n$, use (\ref{phiA}), and perform basic manipulations to get \begin{equation}\label{pumpkin2} |g(x)| \geq \varphi_{1/A,n}^2 \left( \sqrt{ \frac{t_0}{1+t_0}} \right)=:\mathcal{B}(t_0)\quad\text{for all }g\in\mathcal{G}_R.\end{equation} Combining (\ref{pumpkin1}) and (\ref{pumpkin2}), we get that \begin{equation}\label{pumpkin3} \frac{|g(x)|}{|g(y)|}\leq \frac{\mathcal{A}(t_0)}{\mathcal{B}(t_0)} \quad\text{for all }g\in\mathcal{G}_R\text{ and }|x|=|y|=t_0.\end{equation} Since \begin{equation}\left\{\frac{|h(x)|}{|h(y)|}:h\in\mathcal{G}_{R/t_0},|x|=|y|=1\right\} = \left\{\frac{|g(x)|}{|g(y)|}:g\in\mathcal{G}_R,\,|x|=|y|=t_0\right\},\end{equation}
we conclude that \begin{equation}\label{pumpkin5} \max\left\{\frac{|h(x)|}{|h(y)|}:h\in\mathcal{G}_{R/t_0},\,|x|=|y|=1\right\} \leq \frac{\mathcal{A}(t_0)}{\mathcal{B}(t_0)}.\end{equation}
Ideally one would like to choose $t_0$ which minimizes the right hand side of (\ref{pumpkin5}). Unfortunately, this critical value of $t_0$ cannot be solved for algebraically. Instead, following Vuorinen \cite{Vu}, Seittenranta \cite{Se} and Prause \cite{P} we take \begin{equation}\label{t0} t_0 = \left( \lambda_n^{2(\alpha -1) }\frac{A-1}{A} \right)^{\beta}\quad\text{where }\alpha=A^{1/(1-n)},\ \beta=A^{1/(n-1)}.\end{equation}
By Lemma 3.1 in \cite{Se}, if $1<A\leq 2$, then \begin{equation} \frac{\mathcal{A}(t_0)}{\mathcal{B}(t_0)}\leq \exp\left(\left(4\sqrt{2}+\log\frac{1}{A-1}\right)(A^2-1)\right).\end{equation} If $1< A\leq 16/9$, then $4\sqrt{2} \leq 23 \log(1/(A-1))$ and $(A^2-1)\leq 3(A-1)$. Thus, \begin{equation}\label{pumpkin10}\frac{\mathcal{A}(t_0)}{\mathcal{B}(t_0)} \leq \exp\left(72(A-1)\log\frac{1}{A-1}\right)\end{equation} Finally note $72(A-1)\log(1/(A-1))$ is bounded for $1< A\leq 16/9$ and $e^x\leq 1+e^bx$ when $x\leq b$. Therefore, from (\ref{pumpkin5}) and (\ref{pumpkin10}), it readily follows that
\begin{equation}\label{bound3} \max\left\{\frac{|h(x)|}{|h(y)|}:h\in\mathcal{G}_{R/t_0},\,|x|=|y|=1\right\} \leq  1 + C(A-1)\log\left(\frac{1}{A-1}\right)\end{equation} for some absolute constant $C>1$, whenever $1<A\leq 16/9$. \end{proof}

As a model for Theorem \ref{2Klemma}, let us now verify (\ref{bound2}). Suppose that $f:\RR^n\rightarrow\RR^n$ is a $K'$-quasiconformal map (with $K'$ near 1) and assume $f$ is standardized so that $f(0)=0$ and $f(e_1)=e_1$, and $0<|f(x)|<1$ whenever $0<|x|<1$.  Following Vuorinen \cite{Vu}, we fix a point $x$ with $|x|\in(0,1)$ and consider the following four curve families in $\RR^n$ (see Figure 2): \begin{itemize}
\item $\Delta = ([0,x],[e_1,\infty))$,
\item $\Delta^* = ([0,|x|e_1], [e_1,\infty))$,
\item $f(\Delta) = (f[0,x],f[e_1,\infty))$,
\item $f(\Delta)_* = ([-|f(x)|e_1, 0],[e_1, \infty))$.
\end{itemize}
\begin{figure}
\begin{center}\includegraphics[width=.85\textwidth]{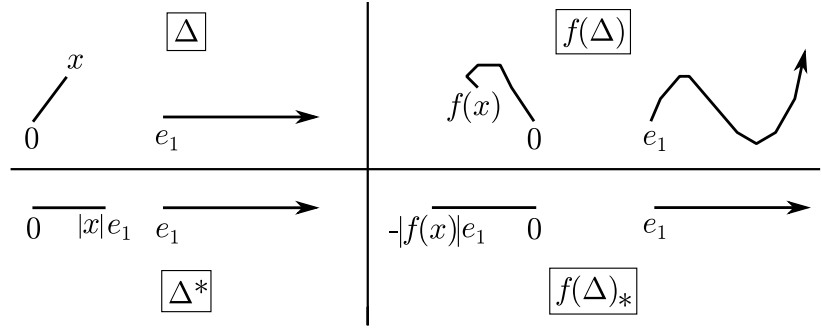}\end{center}
\caption{Families of curves}
\end{figure}
The modulus of these curve families are related by \begin{equation}\label{mark1} \Mod (f(\Delta)_*) \leq \Mod (f(\Delta)) \leq K' \Mod (\Delta) \leq K'  \Mod (\Delta^*),\end{equation} where the first inequality holds by spherical symmetrization (e.g.\ see \S7 in \cite{Vu-book}), the second inequality holds since $f$ is $K'$-quasiconformal, and the third inequality is a lemma of Gehring (see Lemma 5.27 in \cite{Vu-book}).  The rings $[-|f(x)|e_1,0]\cup [e_1,\infty)$ and $[0,|x|e_1]\cup [e_1,\infty)$ used to define $f(\Delta)_*$ and $\Delta^*$, respectively, are conformally equivalent to Teichm\"uller rings by translation and dilation. Thus using the modulus of Teichm\"uller rings we can rewrite (\ref{mark1}) as
\begin{equation}\label{ktau} \tau_n \left( \frac{1}{|f(x)|} \right) \leq K' \tau_n \left( \frac{1}{|x|} -1 \right).\end{equation} Thus, using (\ref{grot}),
\begin{equation}
\label{kgamma}
\gamma_n \left(\sqrt{1+\frac{1}{|f(x)|}} \right) \leq K' \gamma_n \left(\sqrt{\frac{1}{|x|}} \right).
\end{equation} Applying a similar argument with $f^{-1}$ instead of $f$ and $y=f(x)$ instead of $x$ yields
\begin{equation}
\label{kgammainverse}
\gamma_n \left(\sqrt{1+\frac{1}{|x|}} \right) \leq K' \gamma_n \left(\sqrt{\frac{1}{|f(x)|}} \right).
\end{equation}
Notice that $\mathcal{G}=\mathcal{G}(K',K',R)$ is independent of $R>0$. Since (\ref{kgamma}) and (\ref{kgammainverse}) hold for all $f\in\mathcal{G}$ and for all $x$ such that $|x|\in(0,1)$,  Lemma \ref{gammabounds} yields \begin{equation}\max\left\{\frac{|g(x)|}{|g(y)|}:g\in\mathcal{G},\,|x|=|y|=1\right\} \leq 1+C(K'-1)\log\left(\frac{1}{K'-1}\right).\end{equation} Thus, by Remark \ref{howto}, if $f:\RR^n\rightarrow\RR^n$ is $K'$-quasiconformal (with $K'$ near 1), then \begin{equation}\label{mark2} \tH_f(B(w,s))\leq C(K'-1)\log\left(\frac{1}{K'-1}\right) \quad\text{for all }w\in\RR^n\text{ and }s>0.\end{equation} Since the right hand side is independent of $w$ and $s$, (\ref{mark2}) implies (\ref{bound2}).

We will now rerun this argument, with modifications designed to utilize two estimates on the maximal dilatation of $f$.

\begin{proof}[Proof of Theorem \ref{2Klemma}] Let constants $K$ and $K'$ satisfying $1<K\leq\min\{4/3,K'\}$ be given. Fix $R>1$ to be specified later (see (\ref{Rcondition})) and choose $f\in\mathcal{G}(K',K,R)$. Then $f:\RR^n\rightarrow\RR^n$ is a standardized $K'$-quasi\-conformal map and $K_f(B(0,R))\leq K$. Fix $x\in B(0,1)\setminus\{0\}$ and let $\Delta$, $\Delta^*$, $f(\Delta)$ and $f(\Delta)_*$ be the curve families associated to $x$ defined above. Furthermore, decompose $\Delta$ as the union of two curve families, \begin{equation} \Delta=\Gamma_1\cup \Gamma_2,\end{equation} where $\Gamma_1$ consists of all curves in $\Delta$ which remain inside $B(0,R)$ and $\Gamma_2 = \Delta \setminus \Gamma_1$ (see Figure 3).
\begin{figure}\includegraphics[width=.45\textwidth]{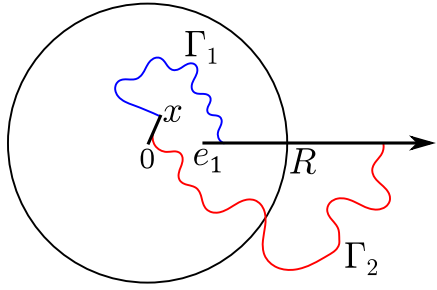}\caption{Curves in $\Gamma_1$ and $\Gamma_2$}\end{figure}
Continuing as above, and using the subadditivity of modulus,
\begin{equation}\label{a1} \Mod(f(\Delta)_*) \leq \Mod(f(\Delta)) \leq \Mod(f(\Gamma_1))+\Mod(f(\Gamma_2)).\end{equation} On one hand, since curves in $\Gamma_1$ lie inside $B(0,R)$, the estimate $K_f(B(0,R))\leq K$ on the maximal dilatation yields \begin{equation}\label{a2} \Mod(f(\Gamma_1)) \leq K \Mod(\Gamma_1)\leq K\Mod(\Delta)\leq K\Mod(\Delta^*).\end{equation} On the other hand, let $\Gamma(r,R)$ be the family of all curves connecting $\partial B(0,r)$ to $\partial B(0,R)$ in $B(0,R)\setminus B(0,r)$. This is one of the few curve families where the modulus is explicitly known (e.g.\ see \cite{Vu-book}): $\rho_n(r,R):=\Mod(\Gamma(r,R))=\sigma_{n-1}(\log R/r)^{1-n}$.  Because every curve in $\Gamma_2$ has a subcurve which belongs to $\Gamma(1,R)$, \begin{equation}\label{a3} \Mod(f(\Gamma_2))\leq K' \Mod(\Gamma_2) \leq K'\Mod(\Gamma(1,R))= K'\rho_n(1,R).\end{equation} Combining (\ref{a1}), (\ref{a2}) and (\ref{a3}) gives \begin{equation}\label{a4} \Mod(f(\Delta)_*)\leq K\Mod(\Delta^*) + K'\rho_n(1,R).\end{equation} Rewriting (\ref{a4}) using the modulus of Gr\"otzsch rings, we get \begin{equation}\label{kkgamma}
\gamma_n \left(\sqrt{1+\frac{1}{|f(x)|}} \right) \leq K \gamma_n \left(\sqrt{\frac{1}{|x|}} \right) + K'\rho_n(1,R).\end{equation} One may view (\ref{kkgamma}) as an analogue of (\ref{kgamma}). We want to find a similar analogue for (\ref{kgammainverse}). We now use the global H\"{o}lder continuity of the quasiconformal map $f$.  Since $f:\RR^n\rightarrow\RR^n$ is $K'$-quasiconformal, $f(0)=0$ and $f(e_1)=e_1$, there exists a constant $M>1$ depending only on $n$ and $K'$ such that \begin{equation}\label{biHolder}M^{-1}\min\{|z|^{\alpha'},|z|^{\beta'}\}\leq |f(z)|\leq M\max\{|z|^{\alpha'},|z|^{\beta'}\}\quad\text{for all }z\in\RR^n\end{equation}
where  $\alpha'=(K')^{1/(1-n)}$ and $\beta'=(K')^{1/(n-1)}$; for this version of H\"{o}lder continuity for normalized quasiconformal maps, see Theorem 1.8(3) in \cite{Vu}. In particular, (\ref{biHolder}) implies that $f(B(0,R))\supset B(0,R^{\alpha'}/M)$. Hence $K_{f^{-1}}(B(0,R^{\alpha'}/M))\leq K$. Thus, if $R>1$ is sufficiently large to ensure $R^{\alpha'}/M>1$, then by arguing as above with $f^{-1}$ instead of $f$ and $y=f(x)$ instead of $x$ we get

\begin{equation}\label{kkgammainverse}
\gamma_n \left(\sqrt{1+\frac{1}{|x|}} \right) \leq K \gamma_n \left(\sqrt{\frac{1}{|f(x)|}} \right) + K'\rho_n(1,R^{\alpha'}/M)\end{equation} for all $x\in B(0,1)\setminus \{0\}$.

Our next task is to choose $R>1$ so large that we can absorb the $\rho_n$-terms in (\ref{kkgamma}) and (\ref{kkgammainverse}) into the $\gamma_n$-terms. Let \begin{equation}\label{t0K2}t_0=\left(\lambda_n^{2(\alpha^2-1)}\frac{K^2-1}{K^2}\right)^{\beta^2}\end{equation} be the constant from Lemma \ref{gammabounds} associated to $A=K^2$, where $\alpha=K^{1/(1-n)}$ and $\beta=K^{1/(n-1)}$. Suppose that we can pick $R>1$ large enough to guarantee \begin{equation}\label{absorb1} K'\rho_n(1,R)\leq K(K-1)\gamma_n \left(\sqrt{\frac{1}{|x|}}\right)\end{equation} and \begin{equation} \label{absorb2} K'\rho_n(1,R^{\alpha'}/M)\leq K(K-1)\gamma_n\left(\sqrt{\frac{1}{|f(x)|}}\right) \end{equation} for all $x$ such that $|x|=t_0$. Then, combining (\ref{kkgamma}), (\ref{kkgammainverse}), (\ref{absorb1}) and (\ref{absorb2}), we see that (\ref{gamma1}) and (\ref{gamma2}) hold with $A=K^2$ for all $f\in\mathcal{G}(K',K,R)$ and for all $x$ such that $|x|=t_0$. Therefore, since $1<K^2\leq 16/9$, Lemma \ref{gammabounds} will imply that \begin{equation}\begin{split} \label{conclusion} \max\left\{\frac{|h(x)|}{|h(y)|}:h\in\mathcal{G}_{R/t_0},\,|x|=|y|=1 \right\}&\leq 1 + C(K^2-1)\log\left(\frac{1}{K^2-1}\right)\\ &\leq 1+3C(K-1)\log\left(\frac{1}{K-1}\right)\end{split}\end{equation} for some absolute constant $C>1$.

Let us now find how large $R>1$ must be to ensure that (\ref{absorb1}) and (\ref{absorb2}) hold. First observe that if $|x|=t_0$ then $\sqrt{1/|f(x)|}\leq (M/t_0^{\beta'})^{1/2}$ by (\ref{biHolder}). Hence, since $\gamma_n$ is decreasing, (\ref{absorb2}) will hold provided that \begin{equation} \label{absorb3} K'\rho_n(1,R^{\alpha'}/M)\leq K(K-1)\gamma_n\left((M/t_0^{\beta'})^{1/2}\right). \end{equation} Using the formula for $\rho_n(1, R^{\alpha'}/M)$ from above and the first inequality in (\ref{grotbound}), we see that (\ref{absorb1}) and (\ref{absorb3}) will hold if
\begin{equation}\label{absorb4a} K'\sigma_{n-1}\left(\log R\right)^{1-n} \leq K(K-1)\sigma_{n-1}\left(\log \lambda_n (1/t_0)^{1/2}\right)^{1-n}\end{equation}
and
\begin{equation}\label{absorb4} K'\sigma_{n-1}\left(\log R^{\alpha'}/M\right)^{1-n}\leq K(K-1)\sigma_{n-1}\left(\log \lambda_n (M/t_0^{\beta'})^{1/2}\right)^{1-n},\end{equation} respectively.
From here an undaunted reader can verify using elementary operations that there is a constant $c>1$ depending only on $n$ and $K'$ so that the inequalities (\ref{absorb4a}) and (\ref{absorb4}) hold whenever \begin{equation}\label{Rcondition} R\geq \left(\frac{c}{K-1}\right)^{c/(K-1)}.\end{equation} For definiteness, let $c>1$ be the smallest constant such that (\ref{Rcondition}) implies (\ref{absorb4}) for all $1<K\leq 4/3$ and set $R=(c/(K-1))^{c/(K-1)}$. From our previous discussion, it follows that (\ref{conclusion}) holds with this choice of $R$. By Remark \ref{howto}, we conclude that $\tH_f(B(w,s))\leq 3C(K-1)\log\left(\frac{1}{K-1}\right)$ for every quasiconformal map $f:\RR^n\rightarrow\RR^n$ such that $K_f(B(w,3(R/t_0)s))\leq K$. Finally observe that \begin{equation} \frac{3R}{t_0} \leq \left(\frac{\tilde c}{K-1}\right)^{\tilde c/(K-1)}=:\widetilde{R}\end{equation} for some constant $\tilde c\geq c$ depending only on $n$ and $K'$. This completes the proof.
\end{proof}

\end{document}